\documentclass[12pt]{amsart}
\usepackage{amssymb}
\usepackage{color}
\usepackage{amsmath,epic,curves}
\usepackage[english]{babel}
\usepackage{graphicx}
\usepackage{comment}
\usepackage{appendix}
\newtheorem{claim}{}[section]
\newtheorem{theorem}[claim]{Theorem}

\newtheorem{lemma}[claim]{Lemma}
\newtheorem{proposition}[claim]{Proposition}

\renewenvironment{proof}{\noindent{\it Proof. \hskip0pt}}
                      {$\square$\par\medskip}

\begin{document}
\baselineskip 6.0 truemm
\parindent 1.5 true pc

\newcommand\lan{\langle}
\newcommand\ran{\rangle}
\newcommand\tr{{\text{\rm Tr}}\,}
\newcommand\ot{\otimes}
\newcommand\ol{\overline}
\newcommand\join{\vee}
\newcommand\meet{\wedge}
\renewcommand\ker{{\text{\rm Ker}}\,}
\newcommand\image{{\text{\rm Im}}\,}
\newcommand\id{{\text{\rm id}}}
\newcommand\tp{{\text{\rm tp}}}
\newcommand\pr{\prime}
\newcommand\e{\epsilon}
\newcommand\la{\lambda}
\newcommand\inte{{\text{\rm int}}\,}
\newcommand\ttt{{\text{\rm t}}}
\newcommand\spa{{\text{\rm span}}\,}
\newcommand\conv{{\text{\rm conv}}\,}
\newcommand\rank{\ {\text{\rm rank of}}\ }
\newcommand\re{{\text{\rm Re}}\,}
\newcommand\ppt{\mathbb T}
\newcommand\rk{{\text{\rm rank}}\,}

\newcommand{\bra}[1]{\langle{#1}|}
\newcommand{\ket}[1]{|{#1}\rangle}

\title{Construction of exposed indecomposable positive linear maps
between matrix algebras}

\author{Kil-Chan Ha and Seung-Hyeok Kye}
\address{Faculty of Mathematics and Statistics, Sejong University, Seoul 143-747, Korea}
\address{Department of Mathematics and Institute of Mathematics, Seoul National University, Seoul 151-742, Korea}

\thanks{KCH and SHK were partially supported by NRFK grant 2014-010499 and  NRFK grant 2013-004942 respectively.}



\begin{abstract}
We construct a large class of indecomposable positive linear maps from the
$2\times 2$ matrix algebra into the $4\times 4$ matrix algebra,
which generate exposed extreme rays of the convex cone of all
positive maps. We show that extreme points of the dual faces for separable states
arising from these maps are parametrized by the Riemann sphere, and the convex hulls of
the extreme points arising from a circle parallel to the equator have the exactly same properties
with the convex hull of the trigonometric moment curve studied from combinatorial topology.
Any interior points of the dual faces are $2\otimes 4$ boundary separable states with full ranks. 
We exhibit concrete examples of such states.
\end{abstract}

\maketitle

\section{Introduction}

Positive linear maps between matrix algebras are now crucial tools in the theory
of quantum information theory, as entanglement witnesses to distinguish 
entanglement from separability. 
We recall that $\varrho\in M_m\otimes M_n$ is separable  if it belongs to $M_m^+\otimes M_n^+$, 
where $M_m^+$ denotes the cone of all positive semi-definite matrices in $M_m$. A state in $(M_m\otimes M_n)^+$ is called entangled if it is not separable.
We denote by $\mathbb S$ and $\mathbb T$ the convex sets of all separable states and states
with positive partial transposes, respectively.
Even though the PPT criterion \cite{choi-ppt,peres}, stating $\mathbb S\subset\mathbb T$,
is quite strong for distinguishing entanglement, we need to construct
indecomposable positive maps between matrix algebras in order to
detect PPT entangled states, by the duality
between positive maps and separable states \cite{eom-kye,horo-1,terhal}.

If we want to detect a set of PPT entangled states with a non-empty interior in $\mathbb T$,
we need to construct an indecomposable positive map which has the bi-spanning property. See
\cite {ha-kye-optimal}.
Furthermore, every interior point $\varrho$ of the dual face of
$\mathbb S$ given by a positive map with the bi-spanning property is again an interior point
of the convex set $\mathbb T$, and so both $\varrho$ and its partial transpose $\varrho^\Gamma$ have the full ranks.
In this context, it is important to note that indecomposable exposed positive maps must have the bi-spanning properties.
See Chapter 8 of \cite{kye_ritsu}. We say that a positive map is exposed (respectively extremal) 
if it generates an exposed (respectively extreme) ray
of the convex cone of all positive maps.
It should be noted that extremity of an indecomposable positive map is not sufficient
for the bi-spanning property.

In spite of its importance, concrete examples of exposed
indecomposable positive maps are very rare in the literature. In the
$3\otimes 3$ case, some variants \cite{cho-kye-lee} of the Choi map
\cite{choi-lam} were proved to be exposed in \cite{ha-kye_exp_Choi}.
On the other hand, it was shown \cite{chrus_Robert} that the
Robertson's map \cite{robertson-83,robertson} between $M_4$ is also
exposed. Some variants in higher dimensional cases have been also constructed
in \cite{sarbicki_exp}. In the $2\ot 4$ case, there is a single
concrete example given by Woronowicz \cite{woronowicz-1} in the
context of non-extendible positive maps. In this $2\otimes 4$ case, even examples
of just indecomposable positive maps are very rare except those in
\cite{tang}, so far the authors know.

This note is an outcome of the author's efforts to understand the above-mentioned example by Woronowicz.
Motivated by this single example, we construct the following linear map $\Phi[a,b,c,d]$ from $M_2$ into $M_4$
which sends
$\left(\begin{matrix}x&y\\z&w\end{matrix}\right)$ to
$$
\left(
\begin{matrix}
hx-cd(y+z)+kw & -gx+gz  &0&0\\
-gx+gy& ax &z &0\\
0&y &bw&-cz-dw\\
0&0&-cy-dw&ex+fw
\end{matrix}
\right),
$$
for given positive numbers $a,b,c$ and $d$ with $ab>1$, where
positive numbers $e,f,g,h$ and $k$ are determined by the relations
\begin{equation}\label{abcd}
(ab-1)\left(\begin{matrix} e\\f\end{matrix}\right)
=a(c+d)\left(\begin{matrix} c\\d\end{matrix}\right),
\quad
g^2=acd,
\quad
h=be-c^2,\quad
k=bf-d^2.
\end{equation}
If $(a,b,c,d)=(2,2,2,1)$, then we get the example of Woronowicz, where
$4d$ should be replaced by $3d$ in the $(1,1)$-entry.
Throughout this note, we denote just by $\Phi$ for the map $\Phi[a,b,c,d]$.

After we briefly summarize known facts on exposedness and bi-spanning property of a positive map in the next section,
we show in Section 3 that the map $\Phi[a,b,c,d]$ is an exposed indecomposable positive linear map. 
It turns out that extreme points of the dual face $\{\Phi\}^\prime$ in $\mathbb S$ are parametrized by the Riemann sphere.
By the abuse of notation, we just denote by $\{\Phi\}^\prime$ for the face $\{\Phi\}^\prime\cap\mathbb S$
of the convex set $\mathbb S$.
In Section 4, we show that circles which are parallel to the equator in the Riemann sphere behave 
in the exactly same way as the
trigonometric moment curve \cite{gale}
$$
\theta\mapsto (e^{i\theta}, e^{2i\theta}, e^{3i\theta}, e^{4i\theta})\in\mathbb C^4=\mathbb R^8.
$$
The convex hulls of extreme points arising from the horizontal circles turn out to be faces of $\{\Phi\}^\prime$.
In Section 5, we show that if we take interior points of two different faces among those
then their convex sums must be boundary separable states with full ranks. This gives an answer
to the problem raised in \cite{chen_full_sep} in the $2\otimes 4$ case.

\section{Exposed positive maps and bi-spanning property}

For a linear map $\phi:M_m\to M_n$, we associate the Choi matrix $C_\phi$ by
$$
C_\phi=\sum_{i,j=1}^m e_{i,j}\ot \phi(e_{i,j})\in M_m\ot M_n,
$$
where $\{e_{i,j}\}$ denotes the usual matrix units of $M_m$. For $\varrho\in M_m\ot M_n$, we define
the bi-linear pairing
$$
\langle\varrho,\phi\rangle
=\tr (\varrho C_\phi^\ttt),
$$
where $C_{\phi}^\ttt$ denotes
the usual transpose of $C_{\phi}$. If $\varrho$ is a pure product state
$|z\rangle\langle z|$ with $|z\rangle =|x\rangle \otimes |y\rangle \in \mathbb C^m\otimes \mathbb C^n$, then we have
\[
\langle |z\rangle\langle z|,\phi\rangle=\langle z| C_{\phi}^{\rm t}|z\rangle =\langle y|\phi(|\bar x\rangle \langle \bar x|)|y\rangle.
\]
The duality tells us that $\varrho$ is separable if and only if
$\langle\varrho,\phi\rangle\ge 0$ for every positive map $\phi$. For a set $P$ of positive maps, we define
the set $P^\prime$ by
$$
P^\prime=\{\varrho\in \mathbb S: \langle\varrho,\phi\rangle =0\ {\text{\rm for every}}\ \phi\in P\},
$$
which gives rise to an exposed face of the convex set $\mathbb S$. It is not known that if every face of $\mathbb S$
is exposed. For a subset $S\subset\mathbb S$, the dual face $S^\prime$ is defined similarly. 
It turns out that $\phi$ is exposed if and only if $\{\phi\}^{\prime\prime}$ is the nonnegative
scalar multiples of $\phi$. If $\phi$ is an exposed then $\{\phi\}^\prime$ is a maximal face of $\mathbb S$, and
every maximal face of $\mathbb S$ is given by $\{\phi\}^\prime$ for a unique exposed map $\phi$. See Chapter 5 of
\cite{kye_ritsu} for details. 

It is not so easy to show that a given positive map is exposed by definition. 
It was shown in \cite{marcin_exp} that the completely positive map $X\mapsto V^*XV$ 
is an exposed positive map. Woronowicz \cite{woro_letter}
showed that a unital irreducible positive map $\phi$ is exposed if the following condition
\begin{equation}\label{dim_cond}
\dim\spa\{a\otimes h\in M_m^+\ot\mathbb C^n: \phi(a)h=0\}=n\times (m^2-1)
\end{equation}
holds. 
This result has been extended \cite{chrus_exposed} for non-unital positive maps. Especially, it was shown that
every irreducible positive map satisfying the condition (\ref{dim_cond}) is exposed whenever
$\phi(I)$ is nonsingular. Recall that a positive map $\phi$ is said to be irreducible if
$[\phi(A), X]=0$ for every $A\in M_m$ implies that $X$ is a scalar multiple of the identity.
We will use this result to show that the positive map $\Phi[a,b,c,d]$ is exposed.

For a positive map $\phi:M_m\to M_n$, we consider the set of product vectors
$$
P[\phi]=\{|z\rangle = |x\rangle\ot |y\rangle \in\mathbb C^m\ot\mathbb C^n: \langle |z\rangle\langle z|,\phi\ran=0\}.
$$
It turns out that $\{\phi\}^{\prime\prime}$ has no completely positive map if and only if the set $P[\phi]$ spans
the whole space $\mathbb C^m\ot\mathbb C^n$. We say that $\phi$ has the spanning property
\cite{lew00} if $\phi$ satisfies these conditions.
We also say that $\phi$ has the bi-spanning property if both $\phi$ and $\phi\circ\ttt$ have the spanning property,
where $\ttt$ denotes the transpose map.
It turns out that the set
$$
\{\varrho\in\mathbb T: \langle\varrho,\phi\rangle <0\}
$$
of PPT entangled states detected by $\phi$ has a nonempty interior
of $\mathbb T$ if and only if the interior of $\{\phi\}^\prime$ lies
on the interior of $\mathbb T$ if and only if $\phi$ has the
bi-spanning property. Further, every exposed indecomposable positive
map has the bi-spanning property. See Chapter 8 of \cite{kye_ritsu}
for details. It should be noted that the Choi map \cite{choi-lam}
has not the bi-spanning property even though it generates
\cite{ha-ext} an extreme ray of the cone of positive maps. See
\cite{ha-kye-optimal}.

We note that a PPT state $\varrho$ is an interior point of $\mathbb T$ if and only if both $\varrho$ and its partial transpose $\varrho^\Gamma$
have full ranks. See Chapter 6 of \cite{kye_ritsu}. Therefore, if we take a positive map $\phi$ with the bi-spanning property
and take an interior point $\varrho$ of the dual face $\{\phi\}^\prime$ then $\varrho$ must be a boundary separable state,
that is, $\varrho$ is a boundary point of $\mathbb S$.
Furthermore, both $\varrho$ and $\varrho^\Gamma$ have full ranks.

\section{Exposed positive maps}

For each complex number $\alpha$, we consider the rank one matrix
$$
P_\alpha=
\left(\begin{matrix}1&\bar \alpha\\ \alpha&|\alpha|^2\end{matrix}\right)
$$
onto the vector $(1,\alpha)^{\rm t}\in\mathbb C^2$,
and consider the determinant $\Delta_i(\alpha)$ of right-below $i\times i$ submatrix of $\Phi(P_\alpha)$
for each $i=1,2,3,4$. By a direct calculation,
we have
$$
\begin{aligned}
\Delta_1(\alpha)&=e+f|\alpha|^2,\\
\Delta_2(\alpha)&=|\alpha|^2\left[h-cd(\alpha+\bar\alpha)+k|\alpha|^2\right],\\
\Delta_3(\alpha)&=acd|\alpha|^2|1-\alpha|^2,\\
\Delta_4(\alpha)&=0.
\end{aligned}
$$
By the choices of $e,f,g,h$ and $k$ in (\ref{abcd}), it is easy to see that
$$
\Delta_1(\alpha)>0,\qquad
\Delta_2(\alpha)>0,\qquad
\Delta_3(\alpha)>0,\qquad
\Delta_4(\alpha)=0
$$
for any complex numbers $\alpha\neq 0,1$. Therefore, we see that
$\Phi(P_\alpha)$ is positive, that is, positive semi-definite for $\alpha\neq 0,1$. For $\alpha=0,1$,
we have
$$
\Phi(P_0)
=\left(\begin{matrix}
h & -g &0 &0\\
-g&a&0&0\\
0&0&0&0\\
0&0&0&e
\end{matrix}\right)
$$
and
$$
\Phi(P_1)
=\left(
\begin{matrix}
h-2cd+k & 0  &0&0\\
0& a &1 &0\\
0&1 &b&-c-d\\
0&0&-c-d&e+f
\end{matrix}
\right).
$$
For $\alpha=\infty$, we also denote by $P_\infty$ the projection onto the vector $|x_{\infty}\rangle :=(0,1)^{\rm t}\in\mathbb C^2$. Then we also have
$$
\Phi(P_\infty)
=\left(\begin{matrix}
k & 0 &0 &0\\
0&0&0&0\\
0&0&b&-d\\
0&0&-d&f
\end{matrix}\right).
$$
It is clear by the above discussion that the $4\times 4$ matrix $\Phi(P_\alpha)$ is a positive 
matrix with rank three
for each $\alpha\in\mathbb C\cup\{\infty\}$.
Therefore, we see that the map $\Phi[a,b,c,d]$ is positive.
We also see that
the kernel space of $\Phi(P_\alpha)$ is spanned by $|y_\alpha\rangle
\in\mathbb C^4$
defined by
\begin{equation}\label{kernel}
|y_\alpha\rangle=\Bigl (g\alpha (1-\alpha),\
\alpha\left[h-cd(\alpha+\bar\alpha)+k|\alpha|^2\right],
-e-f|\alpha|^2,\
-\bar\alpha (c+d\alpha)\Bigr)^{\rm t}\in\mathbb C^4,
\end{equation}
and the kernel of $\Phi(P_\infty)$ is spanned by $|y_\infty\rangle=(0,1,0,0)^\ttt$.

It is straightforward to see that the image of the identity under
$\Phi$ is nonsingular and $\Phi$ is irreducible. We skip the proof.
In order to show that $\Phi$ satisfies the condition (\ref{dim_cond}),
we consider more general situations. For complex functions $f_1,\dots, f_n$, we consider
the vector
\begin{equation}\label{vec}
(f_1(\alpha), f_2(\alpha),\dots, f_n(\alpha))\in\mathbb C^n.
\end{equation}
Then it is clear that the span of these vectors through $\alpha\in\mathbb C$ has
the same dimension as the vector space spanned by functions $f_1,\dots, f_n$.
The following must be well-known. We include here a simple proof
for the convenience of readers.

\begin{lemma}
The set $\{\alpha^k\bar\alpha^\ell: k,\ell=0,1,2,\dots\}$ of monomials is linearly independent.
\end{lemma}

\begin{proof}
Let $P(\alpha)$ be a linear combination of finitely many monomials. Then $P(t\alpha)$
is a polynomial in $t$ whose coefficients are homogeneous polynomials in $\alpha$ and $\bar\alpha$.
Therefore, it suffices to show that homogeneous polynomials are linearly independent.
This can be seen by the identity theorem of analytic functions on the unit circle after 
multiplying $\alpha^n$ to all of homogeneous polynomials of degree $n$.\color{black}
\end{proof}

In order to find the dimension of the span of the vectors in (\ref{vec}) through $\alpha\in\mathbb C$,
it suffices to consider the rank of the \lq coefficient matrix\rq. For example,
the coefficient matrix for $|y_\alpha\rangle$ is given by
$$
\left(
\begin{matrix}
\cdot &g &\cdot &-g &\cdot &\cdot &\cdot &\cdot &\cdot &\cdots\\
\cdot &h &\cdot &-cd &-cd &\cdot &\cdot &k &\cdot &\cdots\\
-e &\cdot &\cdot &\cdot &-f &\cdot &\cdot &\cdot &\cdot &\cdots\\
\cdot &\cdot  &-c  & \cdot &-d &\cdot &\cdot &\cdot &\cdot &\cdots
\end{matrix}
\right)
$$
with respect to monomials $1, \alpha, \bar\alpha, \alpha^2, \alpha\bar\alpha,\bar\alpha^2,\alpha^3,\alpha^2\bar\alpha,\dots$.
This is of rank $4$, which means that the set $\{|y_\alpha\rangle:\alpha\in\mathbb C\}$ spans the whole space $\mathbb C^4$.
The entries of the vector
$$
P_\alpha\ot |y_\alpha\rangle=(1,\alpha,\bar\alpha,\alpha\bar\alpha)^\ttt\ot |y_\alpha\rangle\in\mathbb C^{16}
$$
are linear combinations of the following $12$ monomials
$$
1,\
\alpha,\
\bar\alpha,\
\alpha^2,\
\alpha\bar\alpha,\
\bar\alpha^2,\
\alpha^3,\
\alpha^2\bar\alpha,\
\alpha\bar\alpha^2,\
\alpha^3\bar\alpha,\
\alpha^2\bar\alpha^2,\
\alpha^3\bar\alpha^2.
$$
It is easily checked that the coefficient matrix is of rank $12$. Therefore, we have the following:

\begin{theorem}\label{thm_map}
For positive real numbers $a,b,c$ and $d$ with $ab>1$, the map $\Phi[a,b,c,d]$
is an exposed indecomposable positive linear map from $M_2$ into $M_4$.
\end{theorem}

\begin{proof}
It remains to show that the map $\Phi$ is indecomposable. If it is decomposable, then it must be of the form
$$
X\mapsto V^*XV,\qquad {\text{\rm or}}\qquad X\mapsto V^*X^\ttt V,
$$
since it is exposed. Therefore, the Choi matrix (respectively the partial transpose of Choi matrix) must be of rank one
in the former case (respectively the latter case). But, this is not the case for the map $\Phi$.
\end{proof}

\section{Faces for separable states}

Now, we see that the dual face $\{\Phi\}^\prime$ of $\Phi$ is a maximal face of the convex set $\mathbb S$
whose interior is contained in the interior of $\mathbb T$, and any interior points of $\{\Phi\}^\prime$
are boundary separable states with full ranks. To get concrete examples, we look for specific
faces of $\{\Phi\}^\prime$.

First of all, extreme points of the face $\{\Phi\}^\prime$ are parametrized by the Riemann sphere $\mathbb C\cup\{\infty\}$:
These are pure product states given by pure product vectors
$$
|z_\alpha \rangle =|x_\alpha\rangle\ot|y_\alpha\rangle,\qquad \alpha\in\mathbb C\cup\{\infty\},
$$
where $|x_\alpha\rangle=(1,\bar \alpha)^\ttt,\, |x_\infty\rangle=(0,1)^\ttt\in\mathbb C^2$, and $|y_\alpha\rangle$ is given by (\ref{kernel}).
For a fixed positive number $r>0$, we consider the product vectors which come from the circle $C_r$ centered at the origin with radius $r$:
$$
P_r=\{|z_\alpha\rangle: |\alpha|=r\}.
$$
This represents extreme points of $\{\Phi\}^\prime$ arising from the horizontal circle of the Riemann sphere.
We also denote by $P_r^\Gamma$ the set of partial conjugates $|\bar x_\alpha\rangle \otimes 
|y_\alpha\rangle$'s of product vectors $|x_\alpha\rangle\otimes |y_\alpha\rangle$ in $P_r$.

\begin{lemma}\label{four_vec}
For a fixed $r>0$,
any four mutually distinct vectors $|y_\alpha\rangle \in\mathbb C^4$ with $|\alpha|=r$ are linearly independent.
\end{lemma}

\begin{proof}
We take $\alpha_k=r e^{i \theta_k}$ for real numbers $\theta_k$ with $k=1,2,3,4$, and denote by $M$
the $4\times 4$ matrix whose $k$-row is the transpose of $|y_\alpha\rangle$. Then we have
\[
\text{det}(M)=K e^{i \frac 1 2 (\theta_1+\theta_2+\theta_3+\theta_4)}\prod_{1\le i<j\le 4}\sin\bigr(\dfrac{\theta_i-\theta_j}{2}\bigr)
\]
with
\[
K=\dfrac{64 a c \sqrt {acd}\, r^4(c+d)(c+d\, r^2)\bigr(c(c+d)+d(abc+d)r^2\bigr)}{(ab-1)^2}.
\]
Therefore, $|y_{\alpha_1}\rangle,|y_{\alpha_2}\rangle,|y_{\alpha_3}\rangle,|y_{\alpha_4}\rangle$ are linearly independent for distinct $\alpha_k$'s.
\end{proof}

By Proposition 3.1 in \cite{ha-kye_multi}, we see that
the convex hull of $|z\rangle\langle z|$'s with four mutually distinct product vectors $|z\rangle\in P_r$ is a face of $\{\Phi\}^\prime$
which is affinely isomorphic to the $3$-dimensional simplex with four extreme points.
We also note that
any five mutually distinct product vectors from $P_r$ are linearly independent by Theorem 3.2 \cite{ha-kye_multi}.
Note that all the entries of product vectors in $P_r$ are expressed by polynomials in terms of
$\bar\alpha^2,\bar\alpha, 1,\alpha,\alpha^2$, and so we see that the span of $P_r$ is a $5$-dimensional subspace of
$\mathbb C^2\ot\mathbb C^4$. In the case of $P_r^\Gamma$,
the entries are given by polynomials in $\bar\alpha, 1,\alpha,\alpha^2,\alpha^3$, and so the same is true.

We denote by $\mathbb F_r$ the convex hull of pure product states $|z\rangle\langle z|$ with $|z\rangle\in P_r$.
It is well known that every face of the convex set $\mathbb T$ is of the form
$$
\tau(D,E)=\{\varrho\in\mathbb T: {\mathcal R}\varrho\subset D, {\mathcal R}\varrho^\Gamma\subset E\}
$$
for a given pair $(D,E)$ of subspaces in $\mathbb C^m\ot\mathbb C^n$.
See Section 6 of \cite{kye_ritsu}. The next result shows that the convex set $\mathbb F_r$ is a face of
$\{\Phi\}^\prime\subset \mathbb S$.

\begin{theorem}
The convex set $\mathbb F_r$ coincides with $\tau(\spa P_r,\spa P_r^\Gamma)$.
\end{theorem}

\begin{proof}
We first note that both $\spa P_r$ and $\spa P_r^\Gamma$ are of $5$-dimensional. If $\varrho\in\mathbb T$ is a boundary point 
of $\tau(\spa P_r,\spa P_r^\Gamma)$ then
either $\rk\varrho\le 4$ or $\rk\varrho^\Gamma\le 4$. If this is supported in $\mathbb C^2\ot\mathbb C^3$ then it is separable by
\cite{p-horo,woronowicz}, and if it is supported in $\mathbb C^2\ot\mathbb C^4$ then it is also separable by \cite{2xn}.
Therefore, we see that every PPT state in $\tau(\spa P_r,\spa P_r^\Gamma)$ is separable.

It remains to show that every extreme point in  $\tau(\spa P_r,\spa P_r^\Gamma)$ arises from a product vector in $P_r$.
It suffices to show the following: If a product vector  $|w\rangle=(p_1,p_2)^{\rm t}\otimes (q_1,q_2,q_3,q_4)^{\rm t}$ belongs to $\spa P_r$ and
its partial conjugate $|w\rangle^\Gamma$ belongs to $\spa P_r^\Gamma$, then $|w\rangle=|z_\alpha\rangle$ for $\alpha\in P_r$ up to scalar multiplication.
To do this, we search for a basis of the orthogonal complements of $\spa P_r$:
\begin{equation}\label{pr_perp}
\begin{aligned}
|\zeta_1^r\rangle=&\Bigl(0,\,0,\,\dfrac{d r^2}{c},\,-\dfrac{e+f r^2}{c},\,0,\,0,\,1,\,0\Bigr)^{\rm t},\\
|\zeta_2^r\rangle=&\Bigl(\dfrac{c^2d^2r^2}{g u},\,-\dfrac{cd r^2}{u},\,-\dfrac{r^2\bigl[c^2 u-b(e+fr^2)u-c^2 d^2 r^2\bigr]}{(e+fr^2)u},
\,-dr^2,\,0,\,1,\,0,\,0\Bigr)^{\rm t},\\
|\zeta_3^r\rangle=&\Bigl(\dfrac{cdr^2}{u},\,-\dfrac{gr^2}{u},\,\dfrac{gr^2(u+cdr^2)}{(e+fr^2)u},\,0,\,1,\,0,\,0,\,0\Bigr)^{\rm t}
\end{aligned}
\end{equation}
with $u:=c^2+cd+d^2 r^2-b(e+f r^2)\neq 0$. We also see that
\begin{equation}\label{prg_perp}
\begin{aligned}
|\eta_1^r\rangle=&\Bigl(\dfrac{cd^2r^2}{gu},\,-\dfrac{dr^2}{u},\,-\dfrac{cr^2(u-d^2 r^2)}{(e+fr^2)u},\,0,\,0,\,0,\,0,\,1\Bigr)^{\rm t},\\
|\eta_2^r\rangle=&\Bigl(\dfrac{c d(e+f r^2)}{g u},\,-\dfrac{e+f r^2}{u},\,\dfrac{cdr^2}{u},\,0,\,0,\,0,\,1,\,0\Bigr)^{\rm t},\\
|\eta_3^r\rangle=&\Bigl(-\dfrac{u^2-u c d-c^2 d^2 r^2}{u},\,-\dfrac{u+cdr^2}{u},\,\dfrac{cdr^2(u+cd r^2)}{(e+fr^2)u},\,0,\,-\dfrac{cd}g,\,1,\,0,\,0\Bigr)^{\rm t},
\end{aligned}
\end{equation}
span the orthogonal complement of $P_r^{\Gamma}$.
It remains to solve the equation
$$
\langle \zeta_i^r|w\rangle=0,\qquad
\langle \eta_i^r|w^{\Gamma}\rangle=0,\qquad i=1,2,3,
$$
and one can show that any solution $|w\rangle$ is a scalar multiple of a product vector in $P_r$
by a straightforward calculation.
\end{proof}

We interpret the results in a geometric way.
First of all, the convex set $\mathbb F_r$ is of $8$ affine dimensional. In fact, entries of the pure product state $|z_\alpha\rangle\langle z_\alpha|
\in M_8$ is the linear combinations of $\bar\alpha^{k}$ and $\alpha^k$ with $k=0, 1, 2, 3, 4$, whenever $|\alpha|=r$. Furthermore, we see that
any distinct $9$ choices of $\alpha$'s with $|\alpha|=r$ give rise to linearly independent pure product states by Proposition 3.3 \cite{ha-kye_multi}.
The extreme points of the face $\mathbb F_r$ are parametrized by
the circle. If we take any five extreme points then their convex combination with positive coefficients is an interior point of $\mathbb F_r$. If we take
any $k$ extreme points with $k\le 4$ then their convex hull is a nontrivial face of $\mathbb F_r$, and every face of $\mathbb F_r$ arises
in this way. Therefore, the circle consisting of extreme points behaves
in the exactly same way as the trigonometric moment curve. See \cite{kye_trigono}.
We may consider the circle of the Riemann sphere which are perpendicular to the equator, and get the exactly same results.
We suspect the same is true for any circles in the Riemann sphere.

\section{boundary separable states with full ranks}

In this section, we consider two faces $\mathbb F_r$ and $\mathbb F_s$ for $r\neq s$. First of all, these two faces
have no intersection, because they are faces of a single convex set $\{\Phi\}^\prime$ and have no common extreme points.

\begin{proposition}\label{jgjghj}
For $r\neq s$, the intersection of the $\spa P_r$ and $\spa P_s$ is of $2$-dimensional subspace generated by
the following two product vectors
$$
\zeta_4:=(0,\,1)^{\rm t}\otimes (0,\,0,\,0,\,1)^{\rm t},\quad
\zeta_5:=(1,\,0)^{\rm t}\otimes (g,\, cd,\,0,\,0)^{\rm t}.
$$
On the other hand, $\spa P_r^{\Gamma}\cap \spa P_s^{\Gamma}$ is spanned by
$$
\eta_4:=(0,\,1)^{\rm t}\otimes  (g,\, cd,\,0,\,0)^{\rm t},\quad
\eta_5:=(1,\,0)^{\rm t}\otimes(0,\,0,\,0,\,1)^{\rm t}.
$$
\end{proposition}

\begin{proof}
We note that $\zeta_4$ and $\zeta_5$ are orthogonal to the span of $\zeta^r_i$ for any $r>0$ and $i=1,2,3$.
It is straightforward to see that the following eight vectors
$$
\zeta_1^r,\ \zeta_2^r,\ \zeta_3^r,\ \zeta_1^s,\ \zeta_2^s,\ \zeta_3^s,\ \zeta_4,\ \zeta_5
$$
are linearly independent whenever $r\neq s$. The exactly same thing holds for the intersection $\spa P_r^{\Gamma}$ and $\spa P_s^{\Gamma}$.
\end{proof}

Therefore, we see that the sum of $\spa P_r$ and $\spa P_s$ must be the full space $\mathbb C^2\otimes \mathbb C^4$,
whenever $r\neq s$.
We take an interior point $\varrho_r$ of $\mathbb F_r$ and an interior point $\varrho_s$ of $\mathbb F_s$ then
$$
\varrho=\lambda\varrho_r +(1-\lambda)\varrho_s
$$
has full rank whenever $0<\lambda<1$, and the same is true for its partial transpose $\varrho^\Gamma$. 
Therefore, $\varrho$ is a boundary separable state with full ranks.
More precisely, if we take five points from each of the circles $C_r$ and $C_s$ then the convex combination 
of the corresponding ten pure product states with nonzero coefficients is a boundary separable state with full ranks.
We can go further to take four points from  each of  the circles $C_r$ and $C_s$ to get the separable states with the same properties.
These give us boundary separable states with full ranks whose lengths coincide with the ranks.
We recall that the length of a point $x$ of a convex set is given by the minimum number of extreme points
of which $x$ is a convex combination.

\begin{proposition}
Let $|z_{\alpha_i}\rangle$ {\rm (}respectively $|z_{\beta_i}\rangle${\rm )} be any four mutually distinct vectors from $P_r$ {\rm (}respectively $P_s${\rm )} with $r\neq s$.
Then these eight vectors are linearly independent if and only if the condition holds:
\begin{equation}\label{arg_cond}
\sum_{i=1}^4 \text{\rm arg}(\alpha_i)\neq \sum_{i=1}^4 \text{\rm arg}(\beta_i) \mod 2\pi.
\end{equation}
The same is true for $P_r^\Gamma$ and $P_s^\Gamma$ with $r\neq s$.
\end{proposition}

\begin{proof}
Denote by $\theta_i$ (respectively $\tau_i$) the argument of $\alpha_i$ (respectively, $\beta_i$) for $i=1,2,3,4$.
Let $Q_r$ and $Q_s$ be the subspace spanned by $\{|z_{\alpha_1}\rangle,\,|z_{\alpha_2}\rangle,\,|z_{\alpha_3}\rangle,\,|z_{\alpha_4}\rangle\}$
and $\{|z_{\beta_1}\rangle,\,|z_{\beta_2}\rangle,\,|z_{\beta_3}\rangle,\,|z_{\beta_4}\rangle\}$, respectively.
Then the orthogonal complement of $Q_r$ is of $4$-dimensional subspace generated by
$|\zeta_1^r\rangle,\,|\zeta_2^r\rangle,\,|\zeta_3^r\rangle$ in \eqref{pr_perp} together with the vector $|\zeta_6^r\rangle$ defined by
\[
\Bigl(\dfrac{c\Theta_4 \bigl[u+cd(\Theta_1 r-1) \bigr]}{g u},\,
-\dfrac{c\Theta_4(\Theta_1 r-1)}{u},\,
\dfrac{r^2\Theta_5}{c(e+fr^2)u},
-\dfrac{r(c\Theta_1+dr) }{c},\,0,\,0,\,0,\,1
\Bigr),
\]
where
\[
\begin{aligned}
\Theta_1&:=\sum_{1\le j\le4} e^{-i\theta_j}, & \hskip -4truecm
\Theta_2&:=\sum_{1\le j<k\le 4} e^{-i(\theta_j+\theta_k)},\\
\Theta_3&:=\sum_{1\le j<k<\ell\le 4}e^{-i(\theta_j+\theta_k+\theta_{\ell})}, & \hskip -4truecm
\Theta_4&:=e^{-i(\theta_1+\theta_2+\theta_3+\theta_4)},\\
\Theta_5&:= \bigl(c^3d r\Theta_1 +c^2 u\Theta_2 +cdr u\Theta_3\bigr)\Theta_4 -c^3 d\Theta_4+d^2 r^2 u.
\end{aligned}
\]
If $|w\rangle$ belong to $Q_r\cap Q_s$, then $|w\rangle$ should be orthogonal to all of the vectors $|\zeta_i^r\rangle$
and $|\zeta_i^s\rangle$ for $i=1,2,3,6$.

If $|w\rangle$ belongs to the orthogonal complement of the span of 
$\{|\zeta_i^r\rangle,\,|\zeta_i^s\rangle: i=1,2,3\}$,
then $|w\rangle$ is of the form
\[
\gamma_1\, (0,\,0,\,0,\,0,\,0,\,0,\,0,\,1)^{\rm t}+\gamma_2\, (g,\, cd,\, 0,\,0,\,0,\,0,\,0,\,0)
\]
for some $\gamma_1,\,\gamma_2\in \mathbb C$ by Proposition \ref{jgjghj}. We solve the following equations
\[
\begin{aligned}
\langle \zeta_6^r|w\rangle&=\gamma_1+ c e^{i (\theta_1+\theta_2+\theta_3+\theta_4)} \gamma_2=0,\\
\langle \zeta_6^s|w\rangle&=\gamma_1+ c e^{i (\tau_1+\tau_2+\tau_3+\tau_4)} \gamma_2=0,
\end{aligned}
\]
to conclude that $\gamma_1=\gamma_2=0$ if and only if
the condition (\ref{arg_cond}) holds.
The statement for $P_r^\Gamma$ and $P_s^\Gamma$ can be proved similarly.
\end{proof}

If we take four complex numbers from each of the different circles $C_r$ and $C_s$ with the condition (\ref{arg_cond}) then the
convex combination of their corresponding pure product states with nonzero coefficients 
is a boundary separable state with full ranks.
If we take four complex numbers from a circle and five from another circle then we have the same kinds of separable states.

In the remainder of this note, we also consider the circles $L^\theta$ which are perpendicular to the equator.
They also have the same properties as the family of circles $\{C_r:r>0\}$. We just state the results here.
They are lines
$\arg \alpha=\theta$ in the complex
plane through the origin together with the point $\infty$ in the Riemann sphere, for a given $\theta$. We denote by
$$
P^\theta:=\{|z_\alpha\rangle\in\mathbb C^2\ot\mathbb C^4: \alpha\in L^\theta\}.
$$
It turns out that the convex hull $\mathbb F^\theta$ of $\{|z_\alpha\rangle\langle z_\alpha|: \alpha\in L^\theta\}$ is also a face of $\{\Phi\}^\prime$
which coincides with the face $\tau(\spa P^\theta,\spa (P^\theta)^\Gamma)$ of $\mathbb T$. Both
$\spa P^\theta$ and $\spa (P^\theta)^\Gamma$ are $5$-dimensional. We note that $P^0=(P^0)^\Gamma$ for $\theta=0$.
The intersection of the spans of $P^\theta$ and $P^\tau$ for $\theta\neq\tau$ is the $2$-dimensional space
spanned by $|z_0\rangle$ and $|z_\infty\rangle$, which coincides with the intersection of $\spa (P^\theta)^\Gamma$ and $\spa (P^\tau)^\Gamma$.

For given $\tau\neq \theta$, we take four complex numbers $r_ke^{i\theta}$ from $L^\theta$ and
$s_ke^{i\tau}$ from $L^\tau$ for $k=1,2,3,4$, then the corresponding eight product vectors in $P^\theta\cup P^\tau$ are linearly
independent if and only if $r_1r_2r_3r_4\neq s_1s_2s_3s_4$.
Therefore, we also see that four complex numbers from $L^\theta$ together with five complex numbers from $L^\tau$ for $\tau\neq\theta$
give rise to boundary separable states with full ranks. There is one difference from the family $\{\mathbb F_r: r>0\}$ of faces:
The intersection of two faces $\mathbb F^\theta$ and $\mathbb F^\tau$ with different $\theta$ and $\tau$ is
a line segment with the endpoints $|z_0\rangle\langle z_0|$ and $|z_\infty\rangle\langle z_\infty|$.

It is interesting to note that two circles from different families do not give rise to separable states
with full ranks. For example, the span of product vectors in $P_1\cup P^0$ does not span the whole space.


We close this paper with some comments. The map $\Phi[a,b,c,d]$ which we have constructed has a very special property:
The range of every rank one projection is singular with one-dimensional kernel. 
This property tells us that extreme points of the dual face are parametrized by the Riemann sphere. 
We suspect this property
might have close relations with the notions like exposedness, non-extendibility and/or the condition (\ref{dim_cond})
for indecomposable positive maps, especially when the domain is $2\times 2$ matrices. 
Note that exposed indecomposable positive maps in $M_3$ considered in \cite{ha-kye_exp_Choi}
do not have this property.

In this paper, we found faces $\mathbb F_r$ of $\{\Phi\}^\prime$ whose extreme points are parametrized by circles.
There must be bunch of faces between $\mathbb F_r$ and $\{\Phi\}^\prime$. It would be very interesting to characterize
the whole lattice structures of faces for the convex set $\{\Phi\}^\prime$ which is a maximal face of the whole convex set
$\mathbb S$. Especially,
it is also unclear yet if the interior of the convex hull of $\mathbb P_r$ and $\mathbb P_s$ with $r\neq s$ lies 
in the interior or on the boundary of $\{\Phi\}^\prime$.

The lengths of the $3\otimes 3$ boundary separable states with full ranks constructed 
in \cite{ha-kye_unique} are strictly greater than ranks.
On the other hand, we have constructed same kinds of $2\otimes 4$ separable states whose lengths coincide with ranks. 
We could not determine the lengths of separable states which were given by convex combinations of 
nine or ten pure product states.

\end{document}